\documentclass{siamltex}
\usepackage{amsfonts,amsmath,mathrsfs}
\usepackage{amssymb}
\usepackage{hyperref}
\usepackage{graphicx}
\usepackage{wrapfig}
\usepackage{caption}
\usepackage{subcaption}
\usepackage{float}
\usepackage{bbm}
\usepackage{color}

\newcommand{\N}{\mathbb{N}}
\newcommand{\R}{\mathbb{R}}

\newtheorem{remark}[theorem]{Remark}
\newtheorem{algorithm}{Algorithm}

\title{Efficient inclusion of total variation type priors in quantitative photoacoustic tomography}

\author{A. Hannukainen\footnotemark[2]
\and N. Hyv\"onen\footnotemark[2]
\and H. Majander\footnotemark[3]
\and T. Tarvainen\footnotemark[4]
}

\begin{document}
\maketitle

\renewcommand{\thefootnote}{\fnsymbol{footnote}}

\footnotetext[2]{Aalto University, Department of Mathematics and Systems Analysis, P.O. Box 11100, FI-00076 Aalto, Finland (antti.hannukainen@aalto.fi, nuutti.hyvonen@aalto.fi). The work of NH was supported by the Academy of Finland (decision 267789).}

\footnotetext[3]{Aalto University, Department of Mathematics and Systems Analysis, P.O. Box 11100, FI-00076 Aalto, Finland and \'Ecole Polytechnique, Centre de Math\'ematiques Appliqu\'es, Route de Saclay, 91128 Palaiseau Cedex, France (helle.majander@aalto.fi). The work of this author was supported by the Academy of Finland (decision 267789).}

\footnotetext[4]{University of Eastern Finland, Department of Applied Physics, FI-70211 Kuopio, Finland and University College London, Department of Computer Science, WC1E 6BT London, UK, 
(tanja.tarvainen@uef.fi). The work of this author was supported by the Academy of Finland (decisions 136220, 272803 and the Finnish Centre of Excellence in Inverse Problems Research 250215).}

\begin{abstract}
Quantitative photoacoustic tomography is an emerging imaging technique aimed at estimating the distribution of optical parameters inside tissues from photoacoustic images, which are formed by combining optical information and ultrasonic propagation. 
This optical parameter estimation problem is ill-posed and needs to be approached within the framework of inverse problems. 
Photoacoustic images are three-dimensional and high-resolution. Furthermore, high-resolution reconstructions of the optical parameters are targeted. 
Therefore, in order to provide a practical method for quantitative photoacoustic tomography,  the inversion algorithm needs to be able to perform successfully with problems of prominent size.
In this work, an efficient approach for the inverse problem of quantitative photoacoustic tomography is proposed, assuming an edge-preferring prior for the optical parameters. 
The method is based on iteratively combining priorconditioned LSQR with a lagged diffusivity step and a linearisation of the measurement model, with the needed multiplications by Jacobians performed in a matrix-free manner.
The algorithm is tested with three-dimensional numerical simulations.
The results show that the approach can be used to produce accurate and high quality estimates of absorption and diffusion in complex three-dimensional geometries with moderate computation time and cost.
\end{abstract}

\renewcommand{\thefootnote}{\arabic{footnote}}



\pagestyle{myheadings}
\thispagestyle{plain}
\markboth{A. HANNUKAINEN, N. HYV\"ONEN, H. MAJANDER, AND T. TARVAINEN}{TOTAL VARIATION TYPE PRIORS IN QPAT}

\section{Introduction}
\label{sec:introduction}

In {\em photoacoustic tomography} (PAT), high-contrast, high-res\-o\-lu\-tion images of biological tissue are produced by utilising the photoacoustic effect caused by an externally introduced light pulse. 
Absorption of the light pulse in a target generates an initial acoustic pressure distribution that is proportional to the absorbed energy density of the light. 
Due to the elastic nature of tissue, the initial pressure distribution propagates as an ultrasonic wave and can be measured on the surface of the tissue. 
These measurements are then used to reconstruct the initial pressure and to form images of the target.   
Photoacoustic imaging combines the benefits of optical contrast and ultrasound propagation. 
The optical methods provide information about the distribution of chromophores which are light absorbing molecules within the tissue. 
The chromophores of interest are~e.g.~haemoglobin, melanin and various contrast agents.  
The ultrasonic waves carry this optical information directly to the surface with minimal scattering, thus retaining accurate spatial information as well.
PAT has successfully been applied to the visualisation of different structures in biological tissues such as human blood vessels, microvasculature of tumours and cerebral cortex in small animals. 
For more information about PAT, see~e.g.~\cite{xu2006, li2009, wang2009, beard2011} and the references therein.

{\em Quantitative photoacoustic tomography} (QPAT) is a technique aimed at estimating the concentrations of the chromophores \cite{cox2012a}.  
The inverse problem associated with QPAT is two-fold.
First, the initial acoustic pressure distribution is estimated from the measured acoustic waves. 
This is an inverse initial value problem of acoustics and it has been studied extensively; see~e.g.~\cite{xu2006, kuchment2008, wang2009} and the references therein.  
The second stage of QPAT consists of the optical inverse problem of determining the concentrations of chromophores.   
These concentrations can be reconstructed either by directly estimating them from photoacoustic images obtained at various
wavelengths \cite{cox2009,laufer2010,bal2012,pulkkinen2014} or by first recovering the absorption coefficients at different wavelengths and then calculating the concentrations based on the absorption spectra \cite{razansky2009,cox2009,bal2012}. 
In order to obtain accurate estimates, scattering effects need to
be taken into account \cite{bal2011a,cox2012a,tarvainen2012,pulkkinen2014b}.  
As an alternative to the two-step approach, the estimation of the optical parameters directly from photoacoustic time-series has also been considered \cite{shao2012,song2014,haltmeier2015,gao2015,ding2015,pulkkinen2015b}.

In this work, the optical inverse problem of QPAT is studied assuming the corresponding acoustic inverse problem has already been solved. 
The estimation of absorption and diffusion at a single wavelength of light is considered,
but the extension to multiple wavelengths is straightforward.
Furthermore, it is assumed that the photoacoustic efficiency, which connects the acoustic pressure with the absorbed optical energy density and can be identified with the Gr{\"u}neisen parameter for an absorbing fluid, is known.
For discussion about the estimation of the Gr{\"u}neisen parameter simultaneously with the optical
parameters, see~e.g.~\cite{shao2011,bal2012,naetar2014,pulkkinen2014,alberti2015}. In the optical inverse problem of QPAT, two models of light propagation have been used: the {\em radiative transfer equation} (RTE) \cite{tarvainen2012,saratoon2013,mamonov2014,pulkkinen2015,haltmeier2015} and its {\em diffusion approximation} (DA) \cite{gao2010,bal2011a,shao2011,zemp2010,tarvainen2012,tarvainen2013,naetar2014,pulkkinen2014b,Ren13}.
The forward model used here is based on the DA, although our approach could also be implemented with the RTE.

The optical inverse problem of QPAT is nonlinear and ill-posed. 
In order to overcome the ill-posedness, regularisation or Bayesian methods need to be used \cite{kaipio05}. 
In this work, a Bayesian approach with a Gaussian model for the measurement noise and an edge-preferring prior for the to-be-estimated optical parameters are employed.
To be more precise, we consider total variation \cite{Rudin92} type priors (particularly Perona--Malik~\cite{Perona90}), which are edge-preserving and support piecewise constant images which consist of a few homogeneous levels.
Total variation priors/regularisation have previously been utilised in QPAT in~e.g.~\cite{Gao12,gao2010,bal2011a,tarvainen2012} and a Mumford--Shah type approach in~\cite{Beretta15}.

PAT images are {\em three-dimensional} (3D) and high-resolution, and hence they contain a significant amount of data. 
Furthermore, QPAT aims at high-resolution 3D reconstructions of the optical parameters.  
In consequence, a practical inversion method for QPAT must be able to successfully tackle problems with tens, or even hundreds of thousands of data  and unknowns.
In this work, we propose an efficient algorithm for the optical inverse problem of QPAT, capable of handling edge-preferring, total variation type priors for the optical parameters. The approach is based on iteratively combining a lagged diffusivity step and a linearisation of the measurement model of QPAT with priorconditioned LSQR. The algorithm is a modified version of the one introduced for inverse elliptic boundary value problems in~\cite{Hannukainen15,Harhanen15}; see also \cite{Arridge14} for the original ideas behind the technique. In particular, to facilitate the treatment of far greater amounts of data compared to~\cite{Hannukainen15,Harhanen15}, we implement a matrix-free technique for multiplying vectors by the Jacobian of the measurement map, rendering it possible to painlessly handle (full) Jacobians with, say, $10^5$ rows and columns.
This is one of the few studies where QPAT is investigated in 3D; for previous works see \cite{saratoon2013b,naetar2014,pulkkinen2015b}. In particular, \cite{Gao12,gao2010} studied a gradient-based (bound-constrained) split Bregman method for handling TV regularization in 3D QPAT.

The structure of the paper is as follows. 
The optical measurement model of QPAT is described in Section~\ref{sec:MM}. 
The Bayesian framework is introduced in Section~\ref{sec:Bayes} and the algorithm itself in Section~\ref{sec:algo}.
Section~\ref{sec:numerics} tests the approach with 3D numerical simulations. The conclusions are drawn in Section~\ref{sec:conclusion}.

\section{Measurement model}
\label{sec:MM}
We assume the measurements are static in time and model the examined physical body as a bounded domain $\Omega \subset \R^3$ with a connected complement and a Lipschitz boundary.  The domain $\Omega$ is assumed to be isotropic and the associated diffusion and absorption coefficients are denoted by $\kappa, \mu \in L^\infty_+(\Omega)$, respectively, with the definition
$$
L^\infty_+(\Omega) \, = \, \left\{ v \in L^\infty(\Omega) \, | \, {\rm ess}\inf v > 0 \right\}
$$
that takes into account the positivity of these physical quantities.
In this work the DA of the RTE is used as the model for light transport. 
Compared to the full RTE, the DA generally allows faster reconstruction methods due to its simplicity. 
According to the DA,
the 
photon fluence
$\varphi \in H^1(\Omega)$ corresponding to the photon flux $\Phi \in L^2(\partial \Omega)$ through $\partial \Omega$ satisfies the elliptic Robin boundary value problem
\begin{equation}
\label{fwmod} \left\{
\begin{array}{ll}
-\nabla \cdot (\kappa \nabla \varphi) + \mu \varphi = 0 \qquad  &\text{in} \ \Omega, \\[2mm] 
{\displaystyle \frac{1}{4}} \varphi + {\displaystyle \frac{1}{2}} \nu \cdot \kappa\nabla\varphi = \Phi \qquad &\text{on} \ \partial \Omega,
\end{array} \right.
\end{equation}
where $\nu: \partial \Omega \to \R^3$ is the exterior unit normal of $\partial \Omega$ (cf.~\cite{Grisvard85}). We assume the available photoacoustic measurements are noisy samples of the absorbed energy density $H \in L^2(\Omega)$ defined via
\begin{equation}
\label{measurement}
H = \mu \, \varphi \, ,
\end{equation}
which obviously depends on the optical parameters $\kappa$ and $\mu$.

\begin{lemma}
\label{derivative}
The Fr\'echet derivative of the measurement map
$$
\left[ L^\infty_+(\Omega) \right]^2 \ni (\kappa, \mu) \mapsto H \in L^2(\Omega)
$$
at $(\kappa, \mu) \in [ L^\infty_+(\Omega) ]^2$ is given by the linear mapping
$$
\left[ L^\infty(\Omega) \right]^2 \ni (\vartheta, \theta) \mapsto
\mu \, \varphi' + \theta \, \varphi \in L^2(\Omega)
$$
where $\varphi = \varphi(\kappa,\mu) \in H^1(\Omega)$ is the unique solution of \eqref{fwmod} and $\varphi' = (\varphi'(\kappa,\mu))(\vartheta,\theta) \in H^1(\Omega)$ that of the variational problem
\begin{align}
\label{varphider}
\int_{\Omega}\left( \kappa \nabla \varphi' \cdot \nabla v + \mu \varphi' v \right)& dx  + %2 \gamma
{\displaystyle \frac{1}{2}} \int_{\partial \Omega}   \varphi' v \, d S \nonumber \\[1mm] = & \, - \int_{\Omega} \vartheta \, \nabla \varphi \cdot \nabla v \, dx - \int_{\Omega} \theta \varphi v \, dx
\end{align}
for  all $v \in H^1(\Omega)$.
\end{lemma}

\begin{proof}
It is well known that the map
$$
\left[L^\infty_+(\Omega)\right]^2 \ni (\kappa, \mu) \mapsto \varphi \in H^1(\Omega)
$$
is Fr\'echet differentiable in $[ L^\infty_+(\Omega) ]^2$ and that the corresponding derivative is given by (cf.~e.g.~\cite{Dierkes02})
$$
\left[ L^\infty(\Omega) \right]^2 \ni (\vartheta, \theta) \mapsto
\varphi' \in H^1(\Omega).
$$
Since the bilinear map 
$$
L^\infty(\Omega) \times H^1(\Omega) \ni (\lambda, \psi) \mapsto \lambda \, \psi \in L^2(\Omega)
$$
is obviously continuous, the claim follows from the product and chain rules for Banach spaces. \qquad
\end{proof}

\section{Bayesian framework and the choice of prior}
\label{sec:Bayes}
Let us next consider the discretised version of 
\eqref{fwmod}--\eqref{measurement}. To begin with, we express the 
diffusion
$\kappa$ and the absorption~$\mu$ as exponential quantities,
\begin{equation}
\label{exponential}
\kappa(\tilde{\kappa}) = \kappa_0 \exp (\tilde{\kappa}) \qquad {\rm and} \qquad \mu(\tilde{\mu}) = \mu_0 \exp (\tilde{\mu}),  
\end{equation}
where
\begin{equation}
\label{discr_params}
\tilde{\kappa} = \sum_{n=1}^N \tilde{\kappa}_n \phi_n \qquad \text{and} \qquad \tilde{\mu} = \sum_{n=1}^N \tilde{\mu}_n \phi_n \, 
\end{equation}
are representations of $\tilde{\kappa}$ and $\tilde{\mu}$ with respect to a piecewise linear {\em finite element} (FE) basis  $\{ \phi_n \}_{n=1}^N \subset H^1(\Omega)$ corresponding to a chosen (tetrahedral) partition of $\Omega$.
If there is no possibility of confusion, we denote by $\tilde{\kappa}$ and $\tilde{\mu}$ both the corresponding vectors of coefficients, $\tilde{\kappa},\tilde{\mu} \in \R^N$, and the functions defined in \eqref{discr_params}. The positive real numbers $\kappa_0, \mu_0 > 0$ in \eqref{exponential} are the constant diffusion and absorption levels that produce an energy density that is the most compatible with the available measurements (cf.~Algorithm~\ref{alg:kokohoska} in Section~\ref{sec:algo}). Notice that we have chosen the logarithms of the diffusion and  absorption as the free variables since this automatically guarantees the positivity of the coefficient functions in \eqref{fwmod}.

We assume the available measurement is 
\begin{equation}
\label{meas_mod}
\chi \, = \, h(\tilde{\kappa}, \tilde{\mu}) + \eta \, \in \R^N ,  
\end{equation}
where $h: \R^N \times \R^N \to \R^N$ are the coefficients in an approximation of 
$H(\tilde{\kappa},\tilde{\mu}) = \mu(\tilde{\mu})  \varphi\big(\kappa(\tilde{\kappa}),\mu(\tilde{\mu})\big)$ with respect to the FE basis $\{ \phi_n \}_{n=1}^N$,\footnote{In our numerical experiments, the actual measurement for given target optical coefficients --- which are {\em not} originally represented in the basis $\{ \phi_n \}_{n=1}^N$ --- are simulated by approximating $H$ in a piecewise linear FEM basis on a denser simulation grid, interpolating the obtained node values onto the mesh corresponding to $\{ \phi_n \}_{n=1}^N$, and adding noise. In particular, the model \eqref{meas_mod} is clearly inexact as it does not account for the interpolation step in the data simulation; this corresponds to avoiding an obvious inverse crime.}
\begin{equation}
\label{projH}
H \, \approx \,  \sum_{n=1}^N h_n \phi_n.
\end{equation}
Moreover, $\eta \in \R^{N}$ is a realisation of a normally distributed random variable with zero mean and a known, symmetric and positive definite covariance matrix $\Gamma \in \R^{N \times N}$. It easily follows that the probability density of the measurements given the parameters is
$$
p(\chi \, | \, \tilde{\kappa}, \tilde{\mu}) \, \propto \,
\exp \! \Big(-\frac{1}{2}\big(\chi - h(\tilde{\kappa}, \tilde{\mu})\big)^{\rm T} \Gamma^{-1} \big(\chi - h(\tilde{\kappa}, \tilde{\mu})\big)\Big),
$$
where the constant of proportionality is independent of $\tilde{\kappa}$ and $\tilde{\mu}$. 

The prior information that the optical properties of the examined body are approximately homogeneous apart from clearly distinguishable inhomogeneities is taken into account by 
equipping the logarithms of the diffusion and absorption with the prior densities
\begin{equation}
\label{priors}
p(\tilde{\kappa}) \, \propto \, \exp \! \big( -a \, R(\tilde{\kappa}) \big) \qquad \text{and} \qquad p(\tilde{\mu}) \propto \exp \! \big( -b \, R(\tilde{\mu}) \big)
\end{equation}
where $a,b > 0$ are free parameters and $R$ is of the form
\begin{equation}
\label{aRRa}
R(u) \, = \, \int_\Omega r\big(|\nabla u(x) | \big) \, dx,
\end{equation}
with $r: \R_+ \to \R_+$ being a suitable, continuously differentiable, monotonically increasing function (cf.~e.g.~\cite{Arridge14}). All numerical examples presented in this work are based on a Perona--Malik prior, i.e., 
\begin{equation}
\label{PM}
r(t) = \frac{1}{2}\, T^2 \log\left( 1 + (t/T)^2 \right),
\end{equation}
where $T > 0$ is a small parameter controlling the size of detectable edges~\cite{Perona90}. However, exactly the same reconstruction algorithm could as well be employed in the context of~e.g.~(smoothened) TV or TV$^q$ regularisation by 
simply 
using another choice of $r$ (cf.~\cite{Arridge14,Harhanen15}).

Under the assumption that $\tilde{\kappa}$ and $\tilde{\mu}$ are independent, the Bayes' formula yields
\begin{align*}
p(\tilde{\kappa}, \tilde{\mu} \, | \, \chi) \, &\propto \, p(\chi \, | \, \tilde{\kappa}, \tilde{\mu}) \, p(\tilde{\kappa})p(\tilde{\mu})  \\[1mm] 
&\propto \, \exp \Big(-\frac{1}{2}\big(\chi - h(\tilde{\kappa}, \tilde{\mu})\big)^{\rm T} \Gamma^{-1} \big(\chi - h(\tilde{\kappa}, \tilde{\mu})\big) - a \, R(\tilde{\kappa}) - b \, R(\tilde{\mu}) \Big), 
\end{align*}
where the constants of proportionality do not depend on $\tilde{\kappa}$ and $\tilde{\mu}$. The algorithm described in the following section seeks an (approximate) maximum estimate for this posterior (MAP estimate), or equivalently tries to approximate the minimiser for the %Tikhonov 
functional
\begin{equation}
\label{Tikhonov}
F(\tilde{\kappa}, \tilde{\mu}) \, := \, \frac{1}{2}\big(\chi - h(\tilde{\kappa}, \tilde{\mu})\big)^{\rm T} \Gamma^{-1} \big(\chi - h(\tilde{\kappa}, \tilde{\mu})\big) + a \, R(\tilde{\kappa}) + b \, R(\tilde{\mu}).
\end{equation}
In the following, we denote $\beta = [\tilde{\kappa}^{\rm T}, \tilde{\mu}^{\rm T}]^{\rm T} \in \R^{2N}$ and write occasionally $h(\beta)$ and $F(\beta)$ to shorten the notation.

\begin{remark}
It is arguably unrealistic to assume that $\tilde{\kappa}$ and $\tilde{\mu}$ are independent because both of them are prone to change at interfaces between different tissues. However, since assuming a dependence between the two parameters would make the setting less general as well as less illposed, we leave considerations on introducing a suitable joint prior for future studies.
\end{remark}

\section{The algorithm}
\label{sec:algo}

In this section we briefly introduce our method for minimising \eqref{Tikhonov}; for more information, see \cite{Arridge14,Hannukainen15,Harhanen15}. The algorithm is only described here for a single boundary photon flux $\Phi$, but it trivially generalises to the case of multiple illuminations.

The basic version of the iterative algorithm starts at the initial guess $\beta^{(0)} = [\tilde{\kappa}_\text{init}^{\rm T}, \tilde{\mu}_\text{init}^{\rm T}]^{\rm T}$, where $\tilde{\kappa}_\text{init} = 0 \in \R^{N}$ corresponds to the constant 
diffusion
$\kappa_0$ (cf.~\eqref{exponential}) and $\tilde{\mu}_\text{init} = \log( \chi/(\varphi_0\mu_0))$, with $\varphi_0 = \varphi(\kappa_0, \mu_0)$ being 
the fluence 
corresponding to the homogeneous parameter values $\kappa_0$ and $\mu_0$. Note that the choice of  $\tilde{\mu}_\text{init}$ is motivated by \eqref{measurement} and \eqref{exponential}.
Linearising $h(\beta)$ around $\beta^{(l)}$ in \eqref{Tikhonov} results in a new %Tikhonov 
functional
\begin{equation}
\label{Tikhonovk}
F^{(l)}(\beta) \, := \, \frac{1}{2}\big(y^{(l)} - J^{(l)}\beta )^{\rm T} \Gamma^{-1} \big(y^{(l)} - J^{(l)} \beta \big) +  a \, R(\tilde{\kappa}) + b \, R(\tilde{\mu}),
\end{equation}
where the matrix $J^{(l)} \in \R^{N \times 2N}$ is the Jacobian of the map $\beta \mapsto h(\beta)$ evaluated at $\beta^{(l)}$ and 
$$
y^{(l)} \, = \, \chi - h(\beta^{(l)}) +  J^{(l)} \beta^{(l)} \, \in \R^{N}.
$$
For a given illumination $\Phi \in L^2(\Omega)$, the coefficients $h(\beta^{(l)})$ in \eqref{projH} are solved from \eqref{fwmod}--\eqref{measurement} with the optical parameters defined via \eqref{exponential} by the {\em finite element method} (FEM) employing the aforementioned piecewise linear basis functions $\{\phi_n\}_{n=1}^N$. The approximation of the elements in $J^{(l)}$ is based on Lemma \ref{derivative} 
and the chain rule; the details are given in Appendix. In particular, the Jacobians are not formed explicitly, but the needed matrix-vector multiplications are performed row by row in a matrix-free manner. This enables painless handling of (full) Jacobians with tens of thousands of rows and columns in our numerical experiments. Details about the utilised FE meshes can be found in Section~\ref{sec:numerics}.

Taking the gradient of \eqref{Tikhonovk}, one obtains the necessary condition for a minimiser, that is, 
\begin{equation}
\label{necessary}
(J^{(l)})^{\rm T} \Gamma^{-1} J^{(l)}  \beta  
+ a
\left[
\begin{array}{c}
\!\! (\nabla R)(\tilde{\kappa}) \!\! \\[1mm]
\!\! \mathrm{0} \!\!
\end{array}
\right]
+ b
\left[
\begin{array}{c}
\!\! \mathrm{0} \!\! \\[1mm]
\!\! (\nabla R)(\tilde{\mu}) \!\! 
\end{array}
\right]
 \, = \,  (J^{(l)})^{\rm T} \Gamma^{-1}y^{(l)},
\end{equation}
where $\mathrm{0} \in \R^N$. The gradient $\nabla R: \R^N \to \R^N$ can be given as \cite{Arridge14}
$$
(\nabla R)(u) = M(u) u,
$$
where $M \in \R^{N\times N}$ is the FEM system matrix in the basis $\{\phi_n\}_{n=1}^N$ for the elliptic partial differential operator
\begin{equation}
\label{diffuope}
-\nabla \cdot c_u \nabla \, 
\end{equation}
with a natural boundary condition on $\partial \Omega$ and the positive-valued diffusion coefficient
$$
c_u: x \mapsto \frac{r'(|\nabla u(x)|)}{|\nabla u(x)|}, \quad \Omega \to \R_+.
$$
It follows easily that $M$ is positive semidefinite with the one-dimensional kernel ${\rm Ker}(M) = {\rm span} \{[1, \dots, 1]^{\rm T} \} $.

We rewrite \eqref{necessary} as
\begin{equation}
\label{necessary2}
\left(
(J^{(l)})^{\rm T} \Gamma^{-1} J^{(l)}
+
\left[
\begin{array}{c c}
\!\! a \, M(\tilde{\kappa}) & \mathrm{0} \!\! \\[1mm]
\!\! \mathrm{0} & b \, M(\tilde{\mu}) \!\!
\end{array}
\right] \right) \beta
 \, = \, (J^{(l)})^{\rm T} \Gamma^{-1}y^{(l)},
\end{equation}
and get rid of its nonlinearity with respect to $\beta = [\tilde{\kappa}^{\rm T}, \tilde{\mu}^{\rm T}]^{\rm T}$ by substituting $M(\tilde{\kappa}^{(l)})$ and $M(\tilde{\mu}^{(l)})$ for $M(\tilde{\kappa})$ and $M(\tilde{\mu})$, respectively. This corresponds to a single lagged diffusivity step \cite{Vogel96}. 
Denoting a Cholesky factor of $\Gamma^{-1}$ by $\Gamma^{-1/2}$ and setting
\begin{equation}
\label{kaavoja}
A = \Gamma^{-1/2}J^{(l)}, \qquad M = \left[
\begin{array}{c c}
\!\! \, M(\tilde{\kappa}^{(l)}) & \mathrm{0} \!\! \\[1mm]
\!\! \mathrm{0} & \frac{b}{a} \, M(\tilde{\mu}^{(l)}) \!\!
\end{array} \right], \qquad \tilde{y} = \Gamma^{-1/2}y^{(l)},
\end{equation}
we finally arrive at the equation
\begin{equation}
\label{priorcond1}
\big( A^{\rm T}A + a M \big) \beta \, = \,  A^{\rm T}\tilde{y}, \qquad a > 0,
\end{equation}
from which $\beta^{(l+1)}$ is to be solved.

Solving \eqref{priorcond1} is equivalent to determining the MAP or {\em conditional mean} (CM) estimate for the linear model 
\begin{equation}
\label{linmod}
A \beta \, = \, \tilde{y}
\end{equation}
assuming a suitable additive Gaussian measurement noise model and for $\beta$ a zero-mean (improper) Gaussian prior with a scaled version of $M$ as the inverse covariance matrix. In our setting, the leading idea of priorconditioning \cite{Calvetti07,Calvetti12,Calvetti05,Calvetti08}\footnote{Priorconditioning is related to transforming a Tikhonov functional into the standard form~\cite{Elden82,Hansen98,Hilgers76}.} is to include the prior information in $M$ directly in the Krylov subspace structure produced by LSQR \cite{Paige82b,Paige82a}. As $M$ is only positive semidefinite with a nontrivial kernel, we approximate it in \eqref{priorcond1} by the positive definite matrix
\begin{equation}
\label{delta}
M_{\delta} \, = \, M + \delta I
\end{equation}
where $\delta > 0$ is a small positive constant, that is, we consider
\begin{equation}
\label{priorcond2}
\big( A^{\rm T}A + a M_\delta \big) \beta \, = \,  A^{\rm T}\tilde{y}, \qquad a > 0,
\end{equation}
in place of \eqref{priorcond1}. In terms of the MAP estimate for \eqref{linmod}, this corresponds to assuming that the inverse covariance matrix of the Gaussian prior is proportional to $M_\delta$ instead of $M$, making the prior proper.

We formally introduce a (Cholesky) factorisation $M_\delta = L^{\rm T} L$, but emphasise that such is not actually needed in the final algorithm because we resort to a version of LSQR that is compatible with symmetric preconditioning \cite{Arridge14}.
Subsequently, \eqref{priorcond2} is multiplied from the left by $(L^{-1})^{\rm T}$ and $a$ is chosen to be zero, which altogether leads to the ill-posed linear equation 
\begin{equation}
\label{priorcond}
(L^{-1})^{\rm T} A^{\rm T} A L^{-1} \tilde{\beta} \, = \,  (L^{-1})^{\rm T} A^{\rm T} \tilde{y}
\end{equation}
where $\tilde{\beta} = L \beta$. 
We solve \eqref{priorcond} by combining LSQR \cite{Arridge14} with an early stopping rule; loosely speaking, the regularisation provided by $a>0$ is replaced with the early stopping of a Krylov subspace method. Each round of LSQR includes one multiplication with $M_{\delta}^{-1}$, which is not overly expensive as $M_{\delta}$ results from a discretisation of an elliptic partial differential equation and is, in particular, sparse. If one starts the LSQR iteration from $\tilde{\beta} = 0$, it is easy to see that the approximate solution is in the range of $M_\delta^{-1}$ regardless of the number of iterations; see \cite{Arridge14,Hannukainen15} for more details. As $M_\delta^{-1}$ is proportional to the (fictive) prior covariance matrix for \eqref{priorcond2}, this means that the prior information in $M_\delta$ --- originating from the previous iterates $\tilde{\kappa}^{(l)}$ and $\tilde{\mu}^{(l)}$ of the outer loop, cf.~\eqref{kaavoja} and \eqref{delta} --- is indeed directly included in the candidate solutions for \eqref{priorcond} produced by the LSQR sequence. 

The LSQR iteration is terminated  when the residual for \eqref{linmod} ceases to decrease substantially: We monitor the relative reduction of the residual over a window of $m_0$ LSQR steps, 
\begin{equation}
\label{stop_crit}
r_m = 1 - \frac{| A \beta_{m} - \tilde{y} |}{| A \beta_{m - m_0} - \tilde{y} |}, \qquad  m > m_0,
\end{equation}
where $\beta_{m}$ is the $m$th element in the LSQR sequence. Once $r_{m} \leq \tau$, for some user specified $\tau > 0$ and $m_0 \in \N$, the latest iterate $\beta_{m}$ is named $\beta^{(l+1)}$ and one proceeds to the next linearisation of the measurement model (cf.~\eqref{Tikhonovk}).

Including an overall stopping criterion based on tracking the decrease of the residual for the original nonlinear measurement model, our reconstruction algorithm is altogether as follows:

\begin{algorithm}\label{alg:kokohoska}
Select $T > 0$, the ratio $b/a$ for \eqref{priors}, the parameters $m_0 \in \N$ and $\tau > 0$ related to \eqref{stop_crit}, and $\delta > 0$. 
Let $\mathbf{1} = [1, \dots, 1]^{\rm T} \in \R^N$ and determine $(\kappa_0, \mu_0)$ as the minimising pair for
\begin{equation*}
\big | \Gamma^{-1/2} \big(\mathcal{\chi} - h(\log(\kappa) \mathbf{1}, \log(\mu)  \mathbf{1})\big) \big|
\end{equation*}
over $(\kappa, \mu) \in \R_+^2$. Solve $\varphi_0 = \varphi(\kappa_0 \mathbf{1}, \mu_0 \mathbf{1})$ and initialise $\tilde{\kappa}_{\rm init} = 0 \in \R^{N}$, $\tilde{\mu}_{\rm init} = \log( \chi/(\varphi_0 \mu_0))$.
Set   $\beta^{(0)} = [\tilde{\kappa}_{\rm init}^{\rm T}, \tilde{\mu}_{\rm init}^{\rm T}]^{\rm T}$ and $l=0$. 
\vspace{3mm}
\begin{enumerate}
\item Let $h = h(\beta^{(l)})$, $J = J^{(l)}$ and $\tilde{y} = \Gamma^{-1/2} \big(\chi - h + J \beta^{(l)}\big)$.
(Recall that $J$ is not built explicitly, but the corresponding matrix-vector multiplications are performed in a matrix-free manner as explained in Appendix.)
\vspace{2mm}
\item Build $M(\tilde{\kappa}^{(l)})$ and $M(\tilde{\mu}^{(l)})$ as finite element discretisations of \eqref{diffuope} and form $M_\delta = L^{\rm T} L$ according to \eqref{kaavoja} and \eqref{delta}.
\vspace{2mm}
\item Apply the LSQR algorithm of \cite{Arridge14} to 
$$
(L^{-1})^{\rm T} J^{\rm T} (\Gamma^{-1/2})^{\rm T} \Gamma^{-1/2} J L^{-1}  \tilde{\beta} \, = \,  (L^{-1})^{\rm T} J^{\rm T} (\Gamma^{-1/2})^{\rm T} \tilde{y}, \qquad \beta = L^{-1} \tilde{\beta},
$$
starting from $\tilde{\beta} = 0$. Terminate the iteration when $r_m \leq \tau$ (cf.~\eqref{stop_crit}) and denote the corresponding solution $\beta^{(l+1)}$.

\vspace{2mm}
\item
If the nonlinear residual corresponding to \eqref{meas_mod} has not decreased, i.e., 
$$
\big | \Gamma^{-1/2} \big(\chi - h(\beta^{(l+1)}) \big) \big| \, \geq \, 
\big | \Gamma^{-1/2} (\chi - h ) \big|,
$$
substitute the previous parameter vector $\beta^{(l)} = [(\tilde{\kappa}^{(l)})^{\rm T}, (\tilde{\mu}^{(l)})^{\rm T}]^{\rm T}$ in \eqref{exponential} and declare the resulting $\kappa$ and $\mu$ the reconstruction. Otherwise, set $l \leftarrow l+1$ and return to step~1.
\end{enumerate}
\vspace{3mm}
\end{algorithm}

The performance of Algorithm~\ref{alg:kokohoska} is relatively insensitive to the choice of the free parameters $T, \delta >0$, which are set to $T = 5 \cdot 10^{-3}$ and $\delta = 10^{-6}$ in our numerical experiments. Moreover, we choose $b/a = 1$, which means that we assume as strong priors for $\tilde{\kappa} = \log(\kappa/\kappa_0)$ and $\tilde{\mu} = \log(\mu/\mu_0)$, cf.~\eqref{priors}. The choice of $m_0$ and $\tau$ is a more delicate issue: via trial and error, we ended up setting $m_0 = 10$ and $\tau = 10^{-2}$, that is, all LSQR iterations in our numerical tests are terminated when the residual for \eqref{linmod} decreases less than one percent over ten steps. These are certainly not optimal values for $m_0$ and $\tau$, but they seem rather generic and result in adequate reconstructions.

With $K$ illuminations $\Phi^{(1)}, \dots, \Phi^{(K)} \in L^2(\partial \Omega)$, the number of measurements increases from $N$ to $K N$, which in turn results in larger Jacobians and slower computations. Be that as it may, Algorithm~\ref{alg:kokohoska} trivially generalises to such a setting: one just needs to stack the individual measurements in a vector of length $KN$, form the corresponding covariance matrix for the measurement noise as a block diagonal matrix of the original covariances, and build the `total Jacobian' by piling the `sub-Jacobians' on top of each other (in a matrix-free manner).  The only essential change concerns the initialisation of the absorption coefficient: we form $\exp (\tilde{\mu}_\text{init}^{(k)})$, $k=1,\dots, K$, separately for each illumination as described in Algorithm~\ref{alg:kokohoska} and subsequently choose the exponential of the actual initial guess $\exp(\tilde{\mu}_\text{init})$ to be their average.

\begin{remark}
The early stopping of LSQR in step 3 of Algorithm~\ref{alg:kokohoska} has previously been successfully implemented in the contexts of electrical impedance and optical tomography by resorting to the Morozov discrepancy principle \cite{Harhanen15,Hannukainen15}. In our setting, this would lead to monitoring when the residual for \eqref{linmod} falls below the (whitened) noise level
$$
\sqrt{\mathbb{E}\big(|\Gamma^{-1/2} \eta|^{2}\big)} = \sqrt{KN} .
$$
However, according to our experience, such an approach does not work, in general, for quantitative photoacoustic tomography based on \emph{simulated} interior measurements without committing an inverse crime or using a dubiously large `fudge factor' to scale the noise level: The discrepancy associated to the interpolation of the target energy density $H$, which is computed on a denser simulation grid (cf.~\eqref{projH}), onto the FEM mesh employed in Algorithm~\ref{alg:kokohoska} is easily of the same order as the error corresponding to the one percent of artificial noise that is added to the data in the numerical experiments of Section~\ref{sec:numerics}. This effect is particularly pronounced if the target absorption and diffusion contain jumps that cause quick deviations in $H$ and therefore also lead to larger interpolation errors. 
\end{remark}

\section{Numerical experiments}
\label{sec:numerics}

\begin{figure}[t]
\begin{center}
\includegraphics[width=0.45\textwidth]{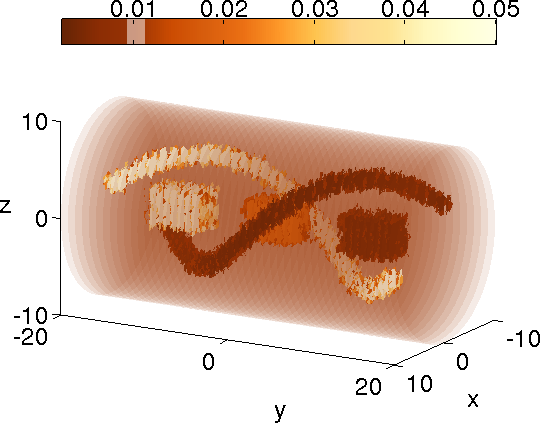} \quad
\includegraphics[width=0.45\textwidth]{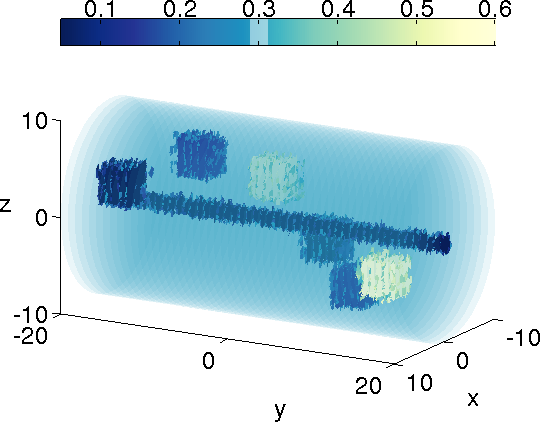} %\\[4mm]
\caption{\label{fig:target} 
Test~1. Left: the target absorption. Right: the target 
diffusion. The values in the intervals $[\mu_\text{bg} \pm 0.001]$ and $[\kappa_\text{bg} \pm 0.01]$ indicated in the colorbars are transparent in the respective images.
}
\end{center}
%\vspace{-2mm}
\end{figure}

To demonstrate the performance of our algorithm, we present two numerical experiments. The first one considers a complicated target and several illuminations. The second test studies how the number and directions of illuminations affect the quality of reconstructions for a somewhat simpler phantom.

\subsection*{Test 1}
We first examine the cylindrical body with constant background properties and embedded inhomogeneities visualised in Figure~\ref{fig:target}. 
The shapes of the target inclusions are described in Table~\ref{tab:target} and the corresponding constant values of the optical parameters are listed in Table~\ref{tab:vals}.
In particular, the target diffusion corresponds to a reduced scattering coefficient that varies in the interval $[0.5, 6.7] \, \text{mm}^{-1}$, which means that our values for the optical parameters could mimic~e.g.~those in breast tissue~\cite{Arridge99}. We illuminate the object in turns with $K = 4$ photon fluxes that penetrate the boundary through rectangular regions located symmetrically around the curved side of the cylinder. The illuminations are homogeneous along the axis of the cylinder, their central polar angles 
are
$\theta_0^{(k)} = (k-1)\pi/2$, $k = 1,\ldots,4$,
and the angular width of each illumination is $\pi/4$. The amplitude of the input flux $\Phi^{(k)}: \partial \Omega \to \R$, $k=1,\dots, 4$, is modeled as
\begin{equation*}
\Phi^{(k)}(\theta) =
\left\{ \begin{array}{ll}
\cos \big( 4(\theta - \theta^{(k)}_0) \big) & \text{if} \ \theta \in \left[\theta^{(k)}_0 - \frac{\pi}{8}, \theta^{(k)}_0 +  \frac{\pi}{8} \right], \\[2mm]
0 & \text{otherwise}
\end{array} \right.
\end{equation*}
where $\theta$ is the polar angle with respect to the axis of the cylindrical domain.

\begin{table}
\begin{center}
\caption{\label{tab:target} 
Test~1. Geometrical specification of the inclusions in the target absorption and diffusion
illustrated in Figure \ref{fig:target}. The unit of length is mm.}
\begin{tabular}{l@{ }l|l@{ }l}
\multicolumn{2}{l|}{Absorption} & \multicolumn{2}{l}{Diffusion} \\ \hline
\multicolumn{2}{l|}{rectangles: size $4 \times 6 \times 4$} &
(1) & cylinder: radius $1$, \\
(1) & center $(0,-11,0)$ & & center $(0,y,0)$, $y \in [-20,20]$ \\
(2) & center $(0,0,0)$ & & \\
(3) & center $(0,11,0)$ & \multicolumn{2}{l}{cubes: size $4^3$,} \\
  &  & \multicolumn{2}{l}{center $(\rho \cos \theta,y,\rho \sin \theta)$, $\rho = 5.5$} \\
\multicolumn{2}{l|}{helical cylinders: radius $1$,} & (2) & $\theta = \pi/6$, $y = -15$ \\
\multicolumn{2}{l|}{center $(\rho \cos \theta,y, \rho \sin \theta)$,} & (3) & $\theta = 3\pi/6$, $y = -9$ \\
\multicolumn{2}{l|}{$\rho = 5.5$, $y \in [-16,16]$} & (4) & $\theta = 5\pi/6$, $y = -3$ \\
(4) & $\theta \in \left[\pi/6, 11 \pi/6 \right]$ & (5) & $\theta = 7\pi/6$, $y = 3$ \\
(5) & $\theta \in \left[7\pi/6, 17 \pi/6 \right]$ & (6) & $\theta = 9\pi/6$, $y = 9$ \\
& & (7) & $\theta = 11\pi/6$, $y = 15$
\end{tabular}
\end{center}
\end{table}

\begin{figure}[t!]
\begin{center}
\includegraphics[width=0.95\textwidth]{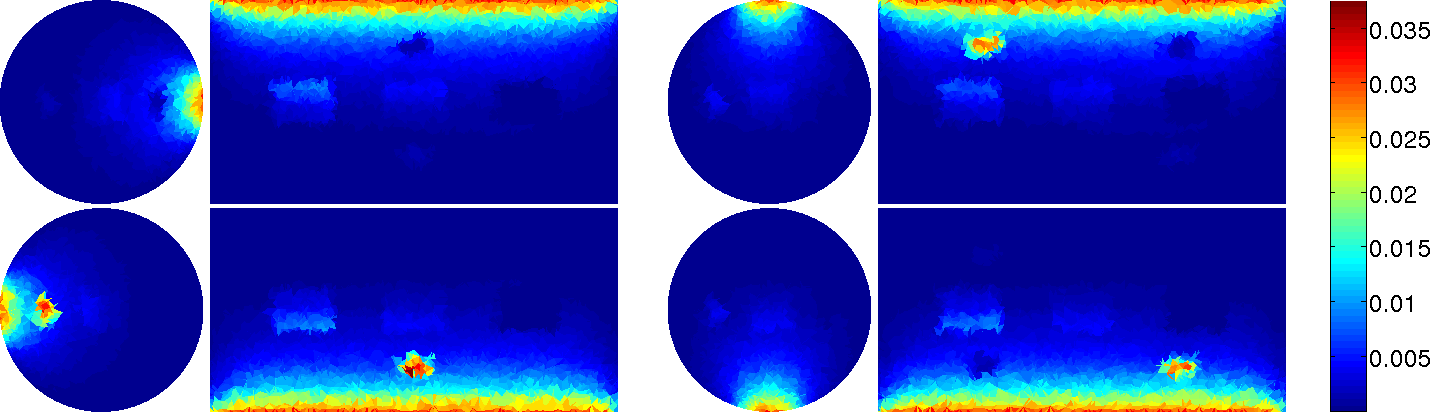}
\caption{\label{fig:meas} 
Test~1. Visualisation of the measurement. Two slices of each $h^{(k)}$, $k = 1, \ldots, 4$, are plotted: one at the position $y=0$, the other at $z = 0$ for $k = 1,3$ (left) and $x = 0$ for $k = 2,4$ (right) depending on the direction of the illumination.
}
\end{center}
%\vspace{-2mm}
\end{figure}

To simulate the data for the numerical experiments, we first solve for each illumination $\Phi = \Phi^{(k)}$ the photon fluence $\varphi^{(k)}$ from the equation \eqref{fwmod} by using a fine FE mesh with $N_\text{f} = 130 091$ nodes and $690 905$ tetrahedrons, and then calculate the corresponding node values of the absorbed energy density $h^{(k)}_{\rm f} \in \R^{N_{\rm f}}$ 
according to \eqref{measurement}. Next, in order to avoid an obvious inverse crime, we project the absorbed energy density onto a coarser FE mesh ($N = 51 794$ nodes and $260 216$ tetrahedrons) which is also used in the reconstruction algorithm. In other words, the (noiseless) measurement, which is illustrated in Figure \ref{fig:meas}, consists of the nodal values
$$
h^{(k)} = P h^{(k)}_{\rm f} \in \R^{N}, \qquad k = 1,\ldots , 4,
$$
where the matrix $P \in \R^{N \times N_\text{f}}$ describes the linear interpolation from the fine FE mesh onto the coarse one.
Finally, we corrupt the measurement with $1\%$ of Gaussian noise, that is, we end up with the data $\chi = \{\chi^{(k)}\}_{k = 1,\ldots,4} \subset \R^N$, where 
$$
\chi^{(k)}_i = h^{(k)}_i + \eta^{(k)}_i
$$
and $\eta^{(k)}_i \sim \mathcal{N}(0, (0.01 \vert h^{(k)}_i \vert)^2)$ for $i = 1, \ldots, N$ and $k=1, \dots, K$.

\begin{figure}[t]
\begin{center}
\includegraphics[width=0.45\textwidth]{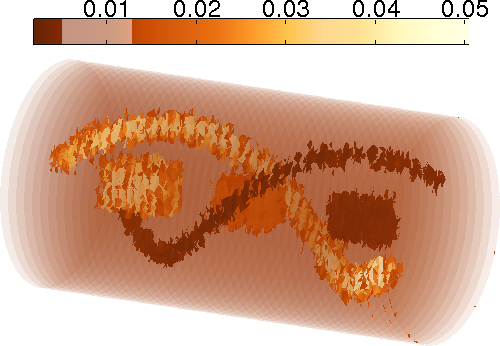} \quad
\includegraphics[width=0.45\textwidth]{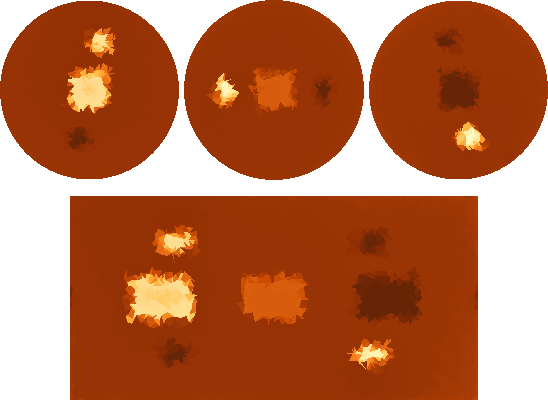} %\\[4mm]
\caption{\label{fig:abs_phase1} 
Test~1. The initial guess $\mu_{\rm init}$ for the absorption. Left: a 3D-visualisation; the values in the interval $[\mu_0 \pm 0.004]$ indicated in the colorbar are transparent. Right: slices at the positions $y = -11, 0, 11$ (top row) and $x = 0$ (bottom row).
}
\end{center}
%\vspace{-2mm}
\end{figure}

In the initialisation phase of Algorithm \ref{alg:kokohoska}, we choose the free parameters as described at the end of Section~\ref{sec:algo},
\begin{equation*}
T = 5 \cdot 10^{-3}, \quad b/a = 1, \quad m_0 = 10, \quad \tau = 10^{-2}, \quad \delta = 10^{-6},
\end{equation*}
and obtain the approximate background values
$\kappa_0 = 0.26 \, \text{mm}$ and $\mu_0 = 0.0087 \, \text{mm}^{-1}$ for the optical parameters. To construct the initial guess for the absorption shown in Figure~\ref{fig:abs_phase1}, we
compute $\varphi^{(k)}_0 = \varphi^{(k)} (\kappa_0 \mathbf{1},\mu_0 \mathbf{1})$ for $k = 1,\ldots,4$ and set $\tilde{\mu}_{\rm init} = \log(\mu_{\rm init}/\mu_0)$, with
\begin{equation*}
\mu_{\rm init} = \frac{1}{4} \sum_{k = 1}^4 \frac{\chi^{(k)}}{\varphi^{(k)}_0}.
\end{equation*}
According to Algorithm~\ref{alg:kokohoska}, the initial guess for the logarithm of the diffusion is $\tilde{\kappa}_{\rm init} = 0 \in \R^{N}$, which corresponds to the homogeneous $\kappa_{\rm init} = \kappa_0 \mathbf{1} \in \R^{N}$ estimate. However, starting from a homogeneous $\kappa_{\rm init}$ leads to somewhat slow convergence during the first rounds of Algorithm~\ref{alg:kokohoska}. To avoid this, we employ the (already quite reasonable) initial guess $\mu_{\rm init}$ for the absorption to run a single LSQR iteration only for the diffusion.
To be more precise, we include a `zeroth step' in Algorithm \ref{alg:kokohoska} in between the initialisation and step 1:
{\em \begin{enumerate}
\setcounter{enumi}{-1}
\item Set $h = h(\tilde{\kappa}_{\rm init},\tilde{\mu}_{\rm init})$. Form the Jacobian $J = J_{\tilde{\kappa}} (\tilde{\kappa}_{\rm init},\tilde{\mu}_{\rm init})$ (w.r.t. $\tilde{\kappa}$ evaluated at $(\tilde{\kappa}_{\rm init},\tilde{\mu}_{\rm init})$), set $\hat{y} = \Gamma^{-1/2} (\chi - h + J \tilde{\kappa}_{\rm init} )$ and build the matrix $M_\delta (\tilde{\kappa}_{\rm init}) = L^{\rm T} L$. Apply the LSQR algorithm to solve
$$
(L^{-1})^{\rm T} J^{\rm T} (\Gamma^{-1/2})^{\rm T} \Gamma^{-1/2} J L^{-1} \hat{\kappa} \, = \,  (L^{-1})^{\rm T} J^{\rm T} (\Gamma^{-1/2})^{\rm T} \hat{y}, \qquad \tilde{\kappa} = L^{-1} \hat{\kappa},
$$
stating from $\hat{\kappa} = 0$ and terminating when $r_m \leq \tau$. (Re)define the initial guess $\tilde{\kappa}_{\rm init}$ to be the corresponding solution. Reset $\beta^{(0)} = [\tilde{\kappa}_{\rm init}^{\rm T}, \tilde{\mu}_{\rm init}^{\rm T}]^{\rm T}$%, $l=0$
and continue to step~1.
\end{enumerate}}
\noindent
The resulting diffusion estimate $\kappa_{\rm init} = \kappa_0 \exp(\tilde{\kappa}_{\rm init})$ is illustrated in Figure~\ref{fig:diff_phase1}. It is not as accurate as the initial guess for the absorption, but inclusions have already started to form at the correct positions.

\begin{figure}[t!]
\begin{center}
\includegraphics[width=0.45\textwidth]{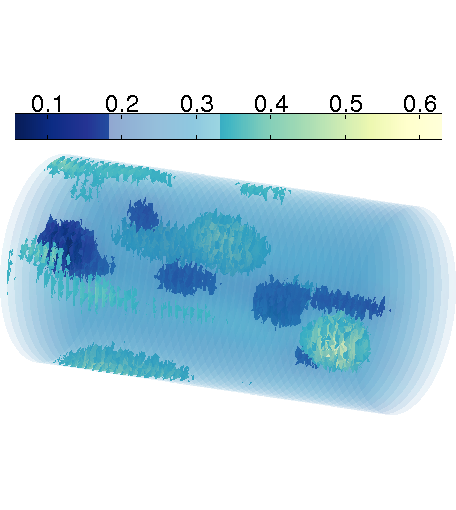} \quad
\includegraphics[width=0.45\textwidth]{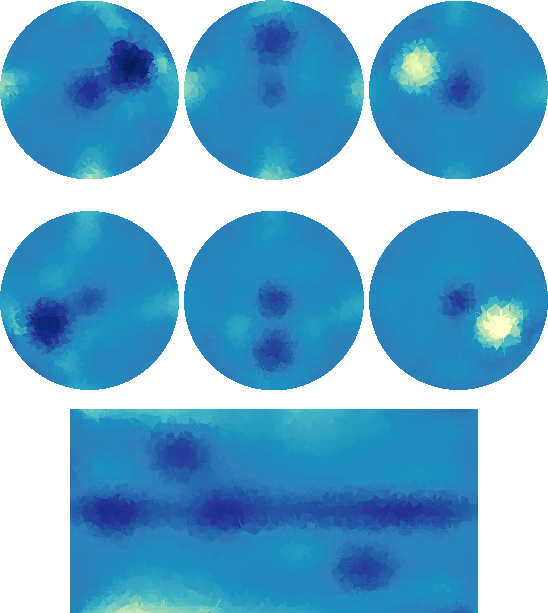} %\\[4mm]
\caption{\label{fig:diff_phase1} 
Test~1. The (refined) initial guess $\kappa_{\rm init}$ for the diffusion.
Left: a 3D-visualisation; the values in the interval $[\kappa_0 \pm 0.075]$ indicated in the colorbar are transparent. Right: slices at the positions $y = -15, -9,-3$ (top row), $y = 3, 9, 15$ (middle row) and $x = 0$ (bottom row).
}
\end{center}
%\vspace{-2mm}
\end{figure}

\begin{figure}[t]
\begin{center}
\includegraphics[width=0.45\textwidth]{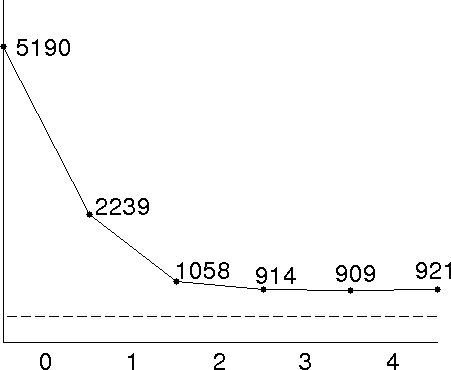} \quad
\includegraphics[width=0.45\textwidth]{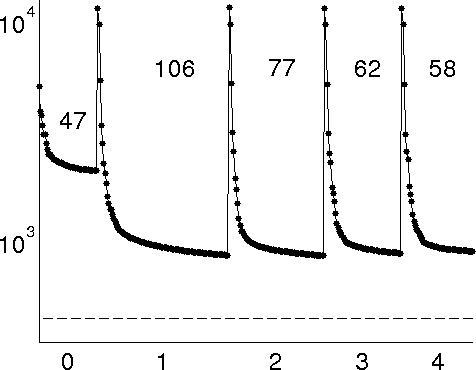} %\\[4mm]
\caption{\label{fig:case1res} 
Test~1. The evolution of the residuals after each iteration step. During step~0, only the diffusion is estimated. The dashed lines indicate the theoretical (whitened) noise level $\sqrt{KN}$.
Left: The nonlinear residuals $| \Gamma^{-1/2} (\chi - h(\beta) ) |$ corresponding to the outer iteration.
Right: The LSQR residuals $| A \beta - \tilde{y} |$ corresponding to the inner iteration on a logarithmic scale.
}
\end{center}
%\vspace{-2mm}
\end{figure}

\begin{figure}[t]
\begin{center}
\includegraphics[width=0.45\textwidth]{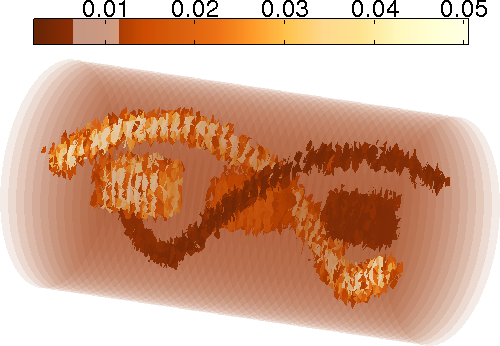} \quad
\includegraphics[width=0.45\textwidth]{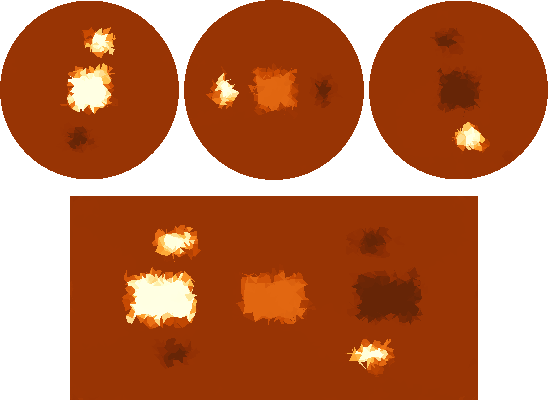} \\[4mm]
\includegraphics[width=0.45\textwidth]{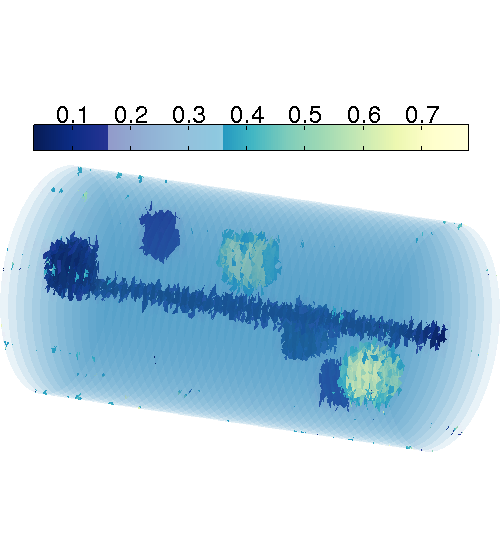} \quad
\includegraphics[width=0.45\textwidth]{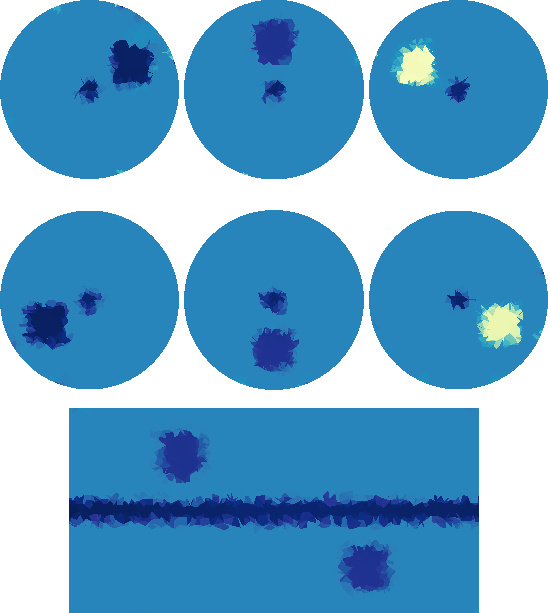}
\caption{\label{fig:case1rec}
Test~1. The reconstructions of the absorption (above) and the diffusion (below). Left: 3D-visualisations; the values in the intervals $[\mu_0 \pm 0.0025]$ and $[\kappa_0 \pm 0.1]$ indicated in the colorbars are transparent in the respective images. Right: slices of the reconstructed absorption at the positions $y = -11, 0, 11$ and $x = 0$ and of the reconstructed diffusion at the positions $y = -15, -9,-3, 3, 9, 15$ and $x = 0$.
}
\end{center}
%\vspace{-2mm}
\end{figure}

The algorithm then proceeds in the standard way, i.e., with the simultaneous reconstruction of the absorption and the diffusion; see steps 1--4 in Algorithm~\ref{alg:kokohoska}.
The residuals for both the outer and the inner iteration are depicted in Figure \ref{fig:case1res}. The algorithm terminates after
five linearisations (including step~0), meaning that the reconstruction, presented in Figure~\ref{fig:case1rec}, is a result of four linearisations of the measurement model (cf.~step~4 of Algorithm~\ref{alg:kokohoska}).
The reconstructions of both optical parameters are in accordance with the prior information: there are well localised inclusions in an approximately constant background. Moreover, the inclusions lie at  approximately correct positions. On the negative side, the reconstruction of the diffusion exhibits some instability near the object boundary, which causes very high jumps at a few isolated nodes: the highest and the lowest point values are $\max (\kappa) = 8.47$ and $\min (\kappa) = 0.0086$, respectively. However, discarding the nodes at the boundary, one gets the reasonable extremal values $\max (\kappa |_{\Omega \setminus \partial \Omega}) = 0.78$ and $\min (\kappa |_{\Omega \setminus \partial \Omega}) = 0.03$ (which are used as the limits for the diffusion plot in Figure~\ref{fig:case1rec}). In addition, in spite of the aforementioned outliers in the reconstructed diffusion, the mean values over the background and the inclusions presented in Table \ref{tab:vals} are approximately correct for both parameters. %Notice also that 
Even though the reconstructed means are not exactly the same as the target values, they are very close to the corresponding mean values of the interpolated parameters $P\mu_{\rm target}$ and $P\kappa_{\rm target}$ (except for the `most difficult' diffusive inclusion (1) lying along the axis of $\Omega$), which is arguably the best one can expect to achieve.

As mentioned after Algorithm~\ref{alg:kokohoska}, the reconstructions seem insensitive to the choice of the threshold parameter $T>0$. Considering significantly higher noise levels (say, 10\%) clearly deteriorates the quality of the reconstructions, with the diffusion coefficient exhibiting a higher level of instability. However, e.g., doubling the noise level does not considerably alter the performance of the algorithm. Moderate changes in the density of the reconstruction mesh mainly affect the reconstruction process via the computation time.

The running time of Algorithm~\ref{alg:kokohoska} for this experiment was approximately $12$ minutes with a MATLAB (2014a) implementation on a laptop with $16$ GB RAM and an Intel Core i7-4600U CPU having clock speed $2.10$ GHz.

\begin{table}
\begin{center}
\caption{\label{tab:vals} 
Test~1. The mean values of the absorption and the diffusion in the target illustrated in Figure \ref{fig:target}, in the target interpolated onto the sparser grid, and in the reconstruction shown in Figure \ref{fig:case1rec}. The mean values are taken over the correct supports of the inclusions listed in Table \ref{tab:target}.
}
\begin{tabular}{l@{\hspace{3mm}}lll|l@{\hspace{3mm}}lll}
Absorption	& \multicolumn{3}{l|}{mean values ($\text{mm}^{-1}$)} & Diffusion & \multicolumn{3}{l}{mean values ($\text{mm}$)} \\
	& $\mu_{\rm target}$ & $P\mu_{\rm target}$ & $\mu_{\rm rec}$  &
	& $\kappa_{\rm target}$ & $P\kappa_{\rm target}$ & $\kappa_{\rm rec}$  \\ \hline
bg	& 0.01 & 0.0101 & 0.00996	& bg  & 0.3 & 0.300 & 0.304 \\
(1)	& 0.05 & 0.0459	& 0.0456		& (1) & 0.05 & 0.0874 & 0.0708 \\
(2)	& 0.02 & 0.0189 & 0.0188		& (2) & 0.05 & 0.0668 & 0.0628 \\
(3)	& 0.002 & 0.00266 & 0.00263	& (3) & 0.15 & 0.163 & 0.165 \\
(4)	& 0.05 & 0.0432 & 0.0425		& (4) & 0.6 & 0.579 & 0.595 \\
(5) & 0.002 & 0.00338 & 0.00335	& (5) & 0.05 & 0.0792 & 0.0764 \\
	& & &						& (6) & 0.15 & 0.169 & 0.174 \\
	& & &						& (7) & 0.6 & 0.573 & 0.576
\end{tabular}
\end{center}
\end{table}

\subsection*{Test 2} In practical applications it is not always possible to illuminate the target from several different directions.
We next study the effect of the number and position of illuminations on the cubical target of size $11^3\, {\rm mm}^3$ visualised in Figure~\ref{subfig:test2target}. 
The absorption is composed of a homogeneous background $\mu_{\rm bg} = 0.015 \, {\rm mm}^{-1}$ with two embedded inhomogeneities: a cross-shaped inclusion lying along the plane $z = x$ with absorption $0.01 \, {\rm mm}^{-1}$ and an origin-centered spherical shell with outer radius $5\, {\rm mm}$, inner radius $4 \, {\rm mm}$ and the absorption level $0.02 \, {\rm mm}^{-1}$ (except at the intersection with the cross). The background diffusion level is $\kappa_{\rm bg} = 0.3 \, {\rm mm}$ and there are also two diffusive inhomogeneities: a cross-shaped inclusion lying along the plane $z = -x$ with diffusion $0.4 \, {\rm mm}$ and a ball centered at the origin with radius $3 \, {\rm mm}$ and diffusion $0.2 \, {\rm mm}$ (except at the intersection with the cross). 

To begin with, we illuminate the object through its bottom face with only one photon flux $\Phi = \Phi^{\rm btm}: \partial \Omega \to \R$ that is modeled as the characteristic function
\begin{equation*}
\Phi^{\rm btm}(x) =
\left\{ \begin{array}{ll}
1 & \text{if} \ x \in \partial \Omega_\text{btm}, \\[1mm]
0 & \text{otherwise},
\end{array} \right.
\end{equation*}
where $\partial \Omega_\text{btm} = \{(x,y,z) \in \R^3 : \vert x \vert \leq 5.5, \vert y \vert \leq 5.5, z = -5.5\}$. The simulation of the data follows the same steps as in the previous example; in particular, the level of additive noise is still one percent. (In this case, the fine FE mesh has $N_f = 133649$ nodes and $752914$ tetrahedrons whereas the coarse one has $N = 54721$ nodes and $295176$ tetrahedrons.)

\begin{figure}[t!]
\begin{center}
\begin{subfigure}[b]{0.41\textwidth}
\includegraphics[width=0.41\textwidth]{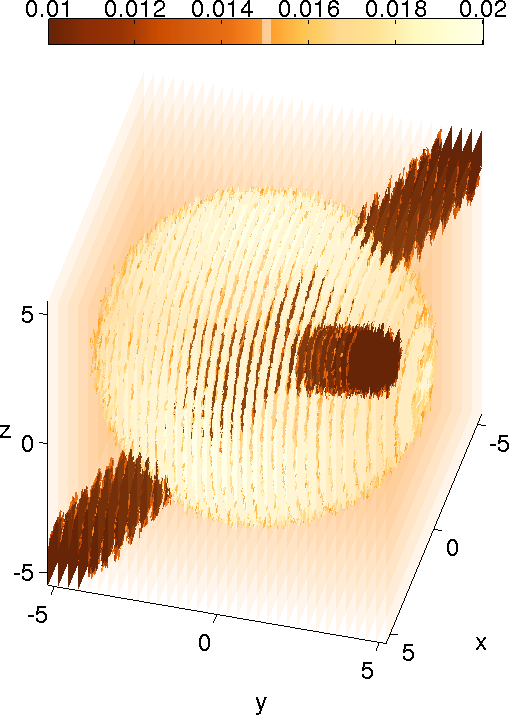} \quad
\includegraphics[width=0.41\textwidth]{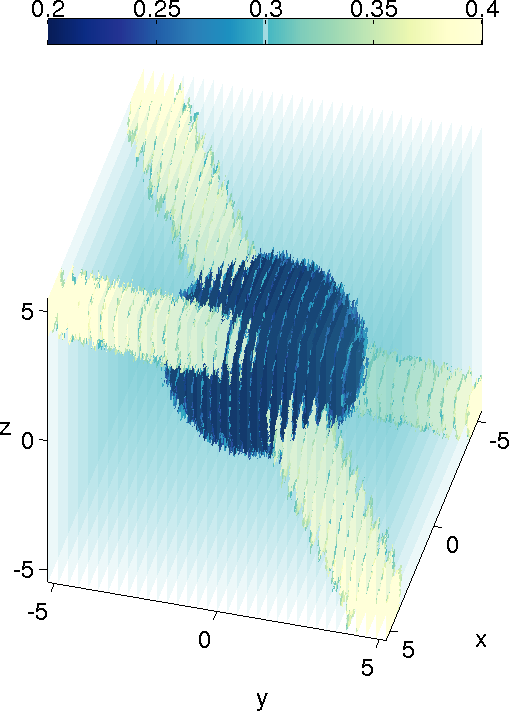} \\
\includegraphics[width=0.41\textwidth]{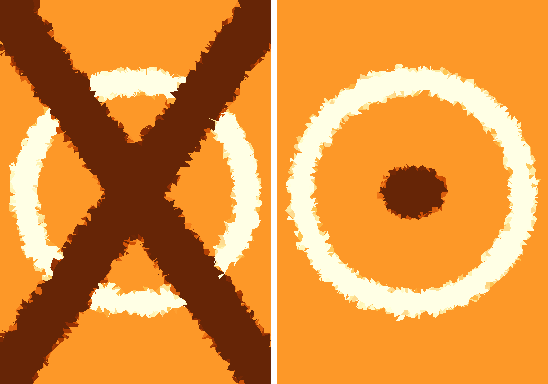} \quad
\includegraphics[width=0.41\textwidth]{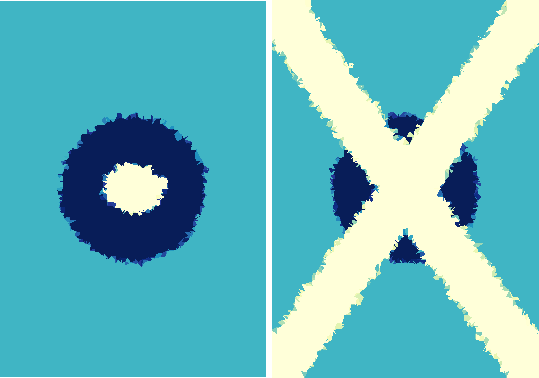}
\caption{\label{subfig:test2target} The target absorption (left) and diffusion (right).}
\end{subfigure} \qquad
\begin{subfigure}[b]{0.41\textwidth}
\includegraphics[width=0.41\textwidth]{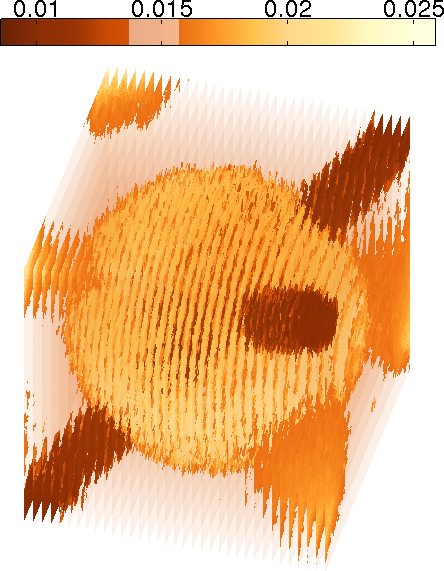} \quad
\includegraphics[width=0.41\textwidth]{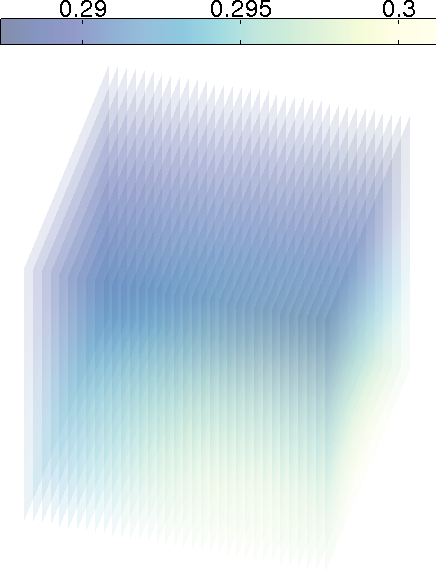} \\
\includegraphics[width=0.41\textwidth]{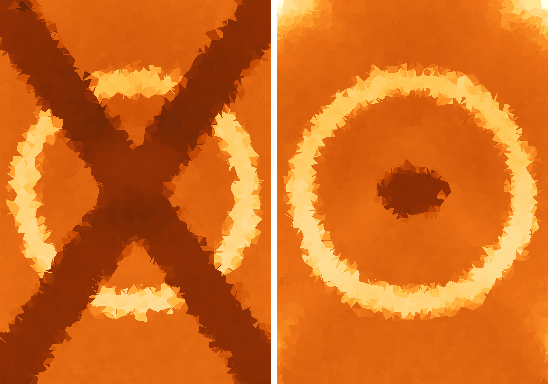} \quad
\includegraphics[width=0.41\textwidth]{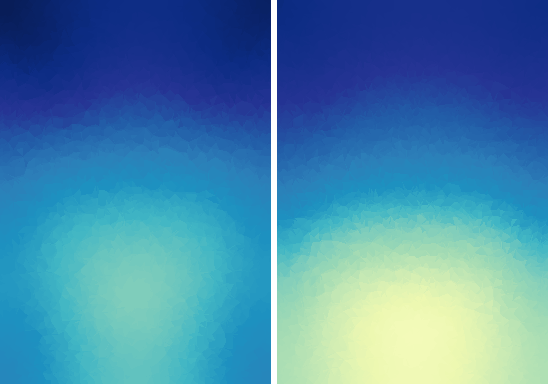}
\caption{\label{subfig:test2a} The reconstruction corresponding to the illumination $\Phi^{\rm btm}$.}
\end{subfigure} \\[5mm]
\begin{subfigure}[b]{0.41\textwidth}
\includegraphics[width=0.41\textwidth]{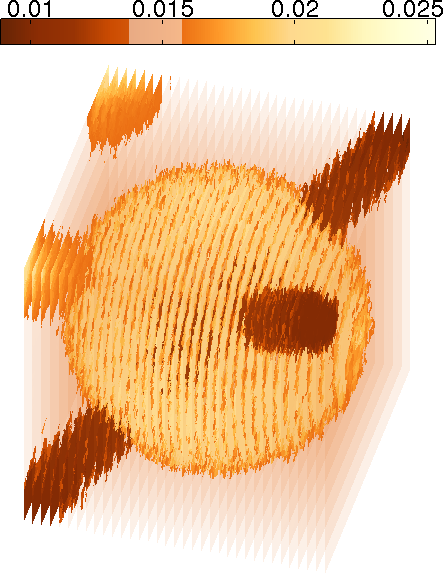} \quad
\includegraphics[width=0.41\textwidth]{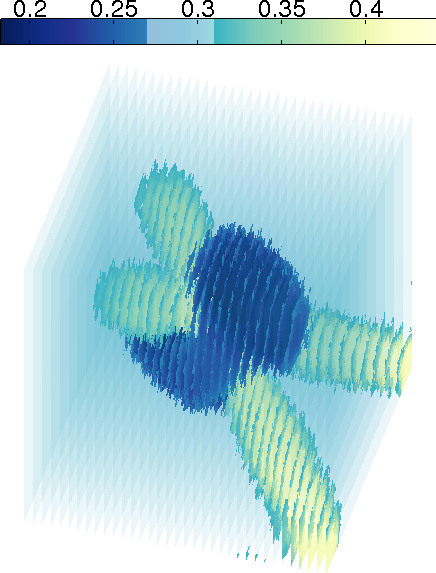} \\
\includegraphics[width=0.41\textwidth]{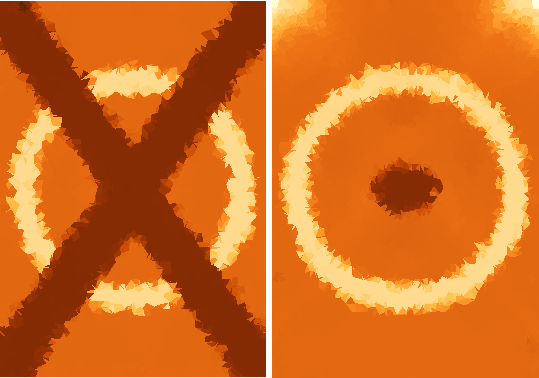} \quad
\includegraphics[width=0.41\textwidth]{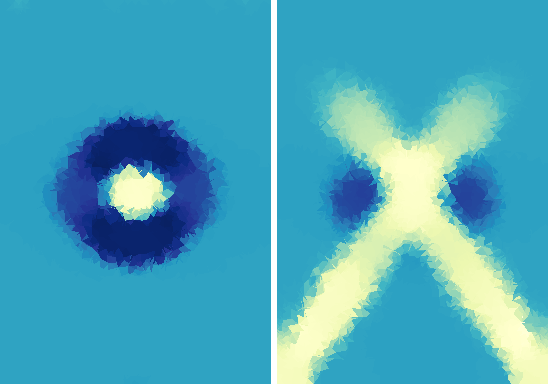}
\caption{\label{subfig:test2b} The reconstruction corresponding to the illuminations $\Phi^{\rm btm}$ and $\Phi^{\rm rgt}$.}
\end{subfigure} \qquad
\begin{subfigure}[b]{0.41\textwidth}
\includegraphics[width=0.41\textwidth]{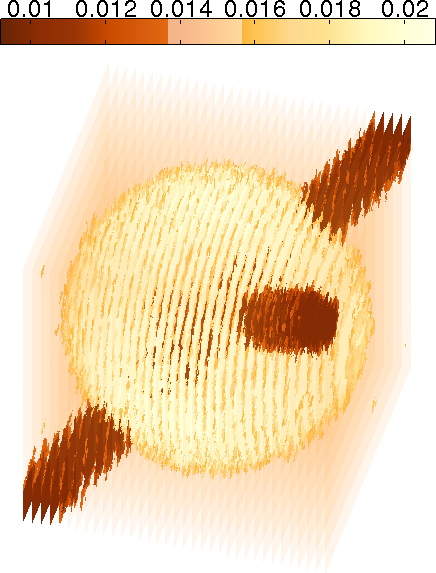} \quad
\includegraphics[width=0.41\textwidth]{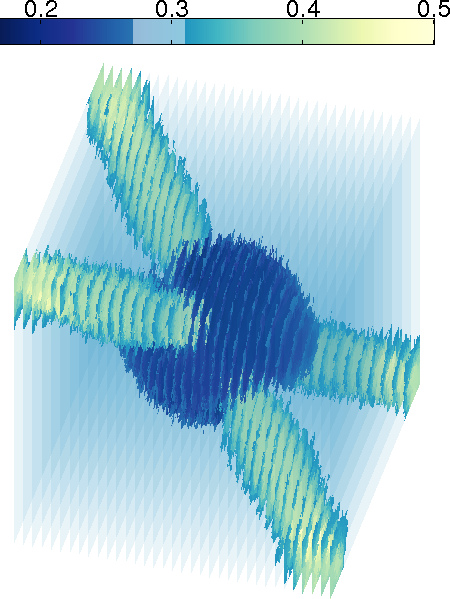} \\
\includegraphics[width=0.41\textwidth]{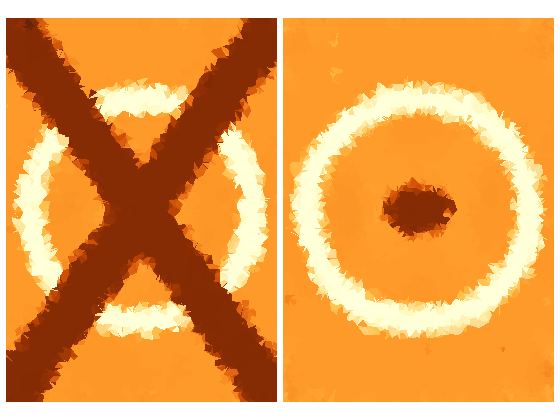} \quad
\includegraphics[width=0.41\textwidth]{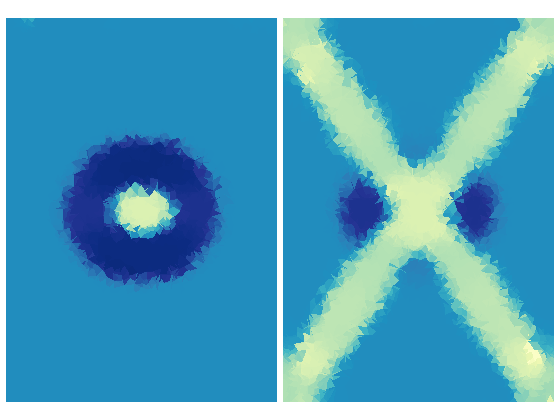}
\caption{\label{subfig:test2c} The reconstruction corresponding to the illuminations $\Phi^{\rm btm}$ and $\Phi^{\rm top}$.}
\end{subfigure}  \\[5mm]
\begin{subfigure}[b]{0.41\textwidth}
\includegraphics[width=0.41\textwidth]{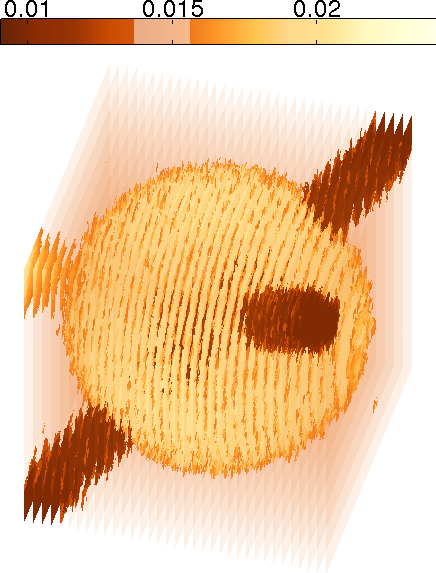} \quad
\includegraphics[width=0.41\textwidth]{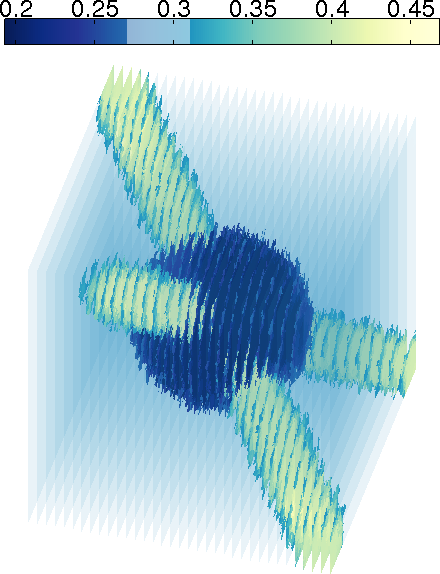} \\
\includegraphics[width=0.41\textwidth]{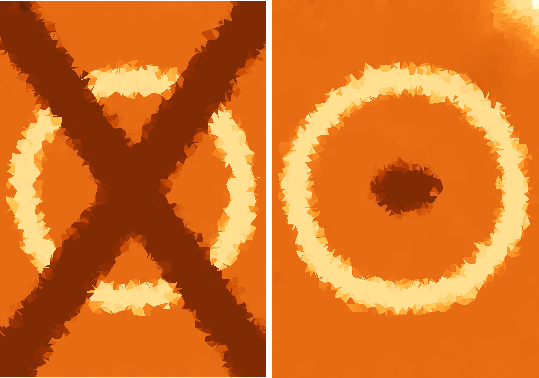} \quad
\includegraphics[width=0.41\textwidth]{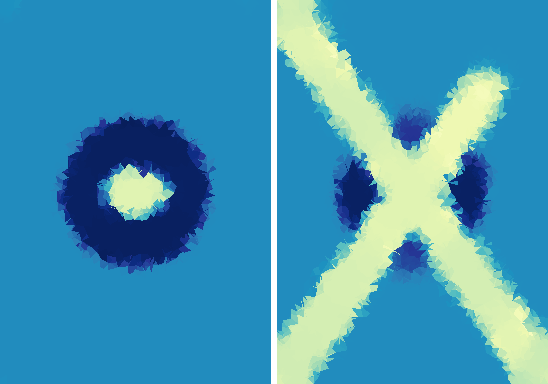}
\caption{\label{subfig:test2d} The reconstruction corresponding to the illuminations $\Phi^{\rm btm}$, $\Phi^{\rm rgt}$ and $\Phi^{\rm bck}$.}
\end{subfigure} \qquad
\begin{subfigure}[b]{0.41\textwidth}
\includegraphics[width=0.41\textwidth]{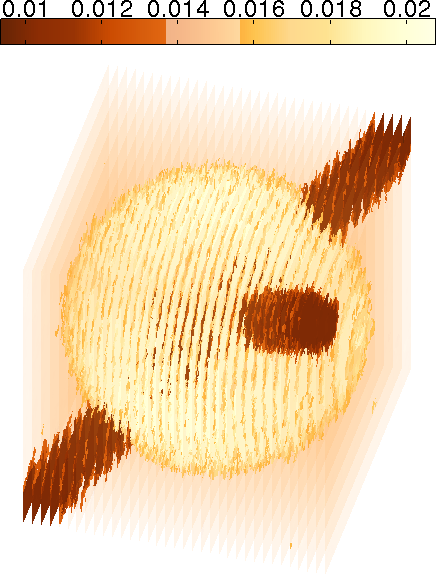} \quad
\includegraphics[width=0.41\textwidth]{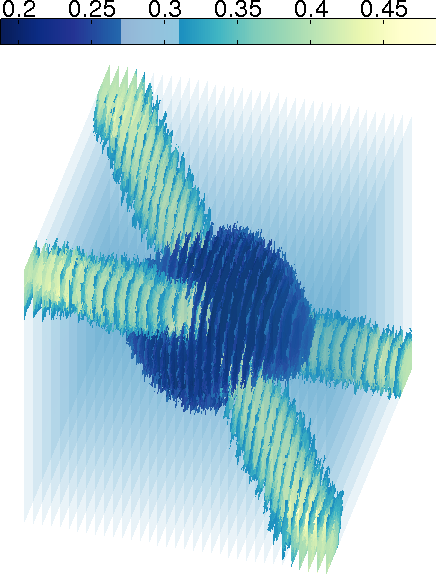} \\
\includegraphics[width=0.41\textwidth]{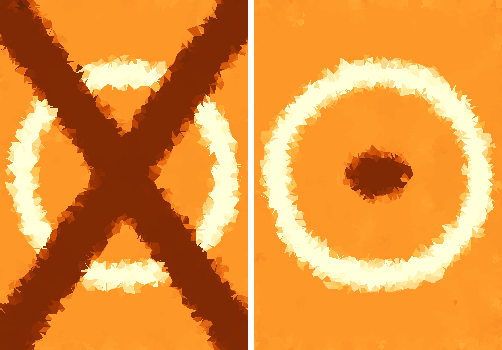} \quad
\includegraphics[width=0.41\textwidth]{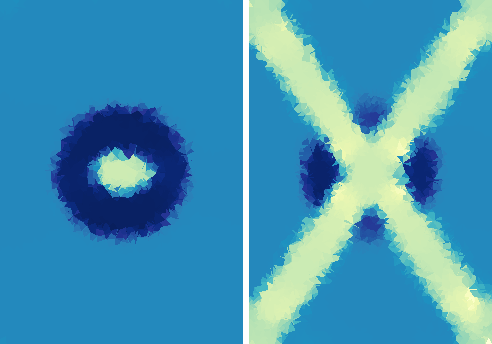}
\caption{\label{subfig:test2e} The reconstruction corresponding to the illuminations $\Phi^{\rm btm}$, $\Phi^{\rm top}$ and $\Phi^{\rm bck}$.}
\end{subfigure}
\vspace{3mm}
\caption{\label{fig:test2} Test~2. In the 3D-visualisations the values in the intervals 
indicated in the colorbars are transparent in the respective images. The slices are taken for both parameters along the planes $z = x$ (left) and $z = -x$ (right).}
\end{center}
\end{figure}

Using the same free parameters and forming the initial guesses $\tilde{\kappa}_{\rm init}$ and $\tilde{\mu}_{\rm init}$ as in the first example, we get the approximate background levels $\kappa_0 = 0.29$ and $\mu_0 = 0.015$ as well as the final reconstruction presented in Figure~\ref{subfig:test2a}.
The reconstruction of the absorption is reasonable, although it includes shadows of the diffusive inclusions. On the other hand, the diffusion produced by Algorithm~\ref{alg:kokohoska} is almost constant,
making it practically useless.  This reconstruction, resulting in a (nonlinear) residual clearly under the noise level $\sqrt{N}$, was obtained after only two steps of LSQR already in step~0 of the algorithm. The reconstructions one gets by starting with $\tilde{\kappa}_{\rm init} = 0$, i.e., without step~0, or even by setting $\tilde{\mu}_{\rm init} = 0$ are essentially the same. In addition, changing the direction of the illumination does not seem to affect the results. We conclude that when using only one illumination, the absorption can explain the measurement (almost) completely, and hence one cannot simultaneously reconstruct the diffusion. 
This is not very surprising as it is well known that the optical inverse problem of QPAT is nonunique for one illumination; see~\cite{bal2011a, cox2009, naetar2014, pulkkinen2015, shao2011}. However, to add an extra twist, \cite{Beretta15} recently presented a uniqueness result for {\em piecewise constant} absorption and diffusion with only one illumination.

Let us then complement the measurement corresponding to the flux $\Phi^{\rm btm}$ with a second one. To evaluate the effect of the chosen illumination directions, we consider two cases: illuminations through adjacent faces and through opposite sides. To study the first case, we combine $\Phi^{\rm btm}$ with $\Phi^{\rm rgt}$,~i.e.,~a homogeneous flux through the right-hand face of the cube. In the second case, $\Phi^{\rm btm}$ is accompanied by $\Phi^{\rm top}$,~i.e.,~a homogeneous flux through the top facet.  Both $\Phi^{\rm rgt}$ and $\Phi^{\rm top}: \partial \Omega \to \R$ are modelled by the appropriate characteristic functions on $\partial \Omega$.
The reconstruction algorithm is run for the two cases with the same free parameters as previously. As a result, %in both cases, 
we get the same approximate background values as with one illumination, but the final reconstructions, presented in Figures~\ref{subfig:test2b}--\ref{subfig:test2c},
are significantly better: two illuminations result in decent reconstructions of both optical parameters. 
It has to be noted, however, that the directions of the two employed photon fluxes play a significant role: for the illuminations through adjacent faces, the reconstructed diffusion is inaccurate close to the opposite edge, whereas the opposite illuminations lead to good quality reconstructions of both parameters in the whole domain. 

Finally, we add a third measurement to the two-illumination settings considered above. The additional measurement is induced by a homogeneous photon flux through the back of the cube, with its amplitude $\Phi^{\rm bck}: \partial \Omega \to \R$ modelled once again by the appropriate characteristic function. 
The triplet $\Phi^{\rm btm}$, $\Phi^{\rm rgt}$ and $\Phi^{\rm bck}$ corresponds to illuminations around one corner of the cube, while the facet-wise fluxes $\Phi^{\rm btm}$, $\Phi^{\rm top}$ and $\Phi^{\rm bck}$ for the other three-illumination test cover the boundary of the cube as evenly as possible. The resulting reconstructions are presented in Figures~\ref{subfig:test2d}--\ref{subfig:test2e}. Compared to the two-illumination cases, the reconstructions of the diffusion are now somewhat sharper. However, as with two adjacent illuminations, when the three photon fluxes are supported around one corner, the reconstruction of the diffusion close to the opposite corner is far from satisfactory. For three illuminations, there should be no problem with the (theoretical) uniqueness, so it seems that the physics of the measurement setting limits the reconstruction quality: the light does not penetrate deep enough into the object in order to provide reliable information about the diffusion near the opposite corner, where the reconstruction mainly reflects the prior information.  On the other hand, complementing the opposite illuminations with a third measurement leads to slightly more detailed reconstructions, but the difference between the two- and three-illumination cases is not substantial (cf.~Figures~\ref{subfig:test2c} and \ref{subfig:test2e}).

The reconstructions corresponding to two illuminations shown in Figures~\ref{subfig:test2b}--\ref{subfig:test2c}
resulted from three linearisations of the measurement model in Algorithm~\ref{alg:kokohoska} (including step~0), while it took four linearisations to produce the three-illumination reconstructions in Figures~\ref{subfig:test2d}--\ref{subfig:test2e}. The corresponding computation times were about five and ten minutes, respectively, with the same hardware as in Test~1.

\section{Concluding remarks}
\label{sec:conclusion}

In this work, the inverse problem of QPAT was investigated.
The aim of QPAT is to produce high-resolution reconstructions of the optical parameters of interest from given three-dimensional, high-resolution PAT images.  
A computationally efficient algorithm for the inverse problem of QPAT was introduced by combining priorconditioned LSQR, a lagged diffusivity step and linearisations of the measurement model in two nested iterations. To facilitate handling a high number of measurements and unknowns, all multiplications by Jacobians in the algorithm were implemented in a matrix-free manner.
The numerical studies exclusively employed a Perona--Malik prior, although other types of edge-preserving priors can as easily be utilised in the described approach.
The proposed reconstruction algorithm was tested with 3D numerical simulations.
The results demonstrate that the approach is capable of producing accurate and good quality estimates of the absorption and diffusion in complex 3D geometries in a reasonable computation time, even when the tests are run on a standard laptop. 
This suggests that QPAT can be developed into a practical method without compromising the accuracy of the estimates or the good resolution the method can potentially provide.

\section*{Appendix: Matrix-free approximation of the Jacobians}

The aim of this appendix is explaining how to efficiently approximate multiplications by the Jacobian $J^h_{\tilde{\kappa}, \tilde{\mu}}$ of the map 
$$
\R^{2N} \ni (\tilde{\kappa}, \tilde{\mu}) \mapsto h\big(\kappa(\tilde{\kappa}), \mu(\tilde{\mu})\big) \in \R^N,
$$ 
where $h$ is the discretisation of the measurement $H(\kappa(\tilde{\kappa}), \mu(\tilde{\mu})) = \mu(\tilde{\mu}) \varphi\big(\kappa(\tilde{\kappa}),\mu(\tilde{\mu})\big)$, cf.~\eqref{projH}, and \begin{equation*}
\kappa(\tilde{\kappa}) = \kappa_0 \exp (\tilde{\kappa}) \qquad {\rm and} \qquad \mu(\tilde{\mu}) = \mu_0 \exp (\tilde{\mu})
\end{equation*}
are the elementwise transformations of the parameters introduced in \eqref{exponential}.
As in Sections~\ref{sec:Bayes} and \ref{sec:algo}, we only consider here the case of one illumination, which corresponds to the photon flux~$\Phi$, cf.~\eqref{fwmod}. In the case of multiple illuminations, more bookkeeping of the indices is required, 
but otherwise the described methodology works as for a single flux of photons.

First of all, since 
$\kappa_n = \kappa(\tilde{\kappa}_n)$ and $\mu_n = \mu(\tilde{\mu}_n)$ for all $n = 1,\ldots,N$, 
we get the Jacobian $J^h_{\tilde{\kappa}, \tilde{\mu}} = [ J^{h}_{\tilde{\kappa}} , J^{h}_{\tilde{\mu}} ] \in \R^{N \times 2N}$ by the chain rule
\begin{align*}
J^{h}_{\tilde{\kappa}} &= \left[ \frac{\partial h_i}{\partial \tilde{\kappa}_j} \right]_{i,j=1}^N = \left[ \frac{\partial h_i}{\partial \kappa_j} \frac{d \kappa_j}{d \tilde{\kappa}_j} \right]_{i,j=1}^N = \left[ \frac{\partial h_i}{\partial \kappa_j} \kappa_j \right]_{i,j=1}^N = J^{h}_{\kappa} \diag (\kappa), \\
J^{h}_{\tilde{\mu}} &= \left[ \frac{\partial h_i}{\partial \tilde{\mu}_j} \right]_{i,j=1}^N = \left[ \frac{\partial h_i}{\partial \mu_j} \frac{d \mu_j}{d \tilde{\mu}_j} \right]_{i,j=1}^N = \left[ \frac{\partial h_i}{\partial \mu_j} \mu_j \right]_{i,j=1}^N = J^{h}_{\mu} \diag (\mu).
\end{align*}
In particular, operating with $J^h_{\tilde{\kappa}, \tilde{\mu}}$ boils, in essence, down to multiplying with $J^h_{\kappa, \mu} = [ J^{h}_{\kappa}, J^{h}_{\mu} ]$.

According to Lemma \ref{derivative}, the Fr\'echet derivative of $H$ is
\begin{align*}
\big(H'(\kappa,\mu) \big) (\vartheta, \theta) = \mu \varphi' + \theta \varphi,
\end{align*}
where $\varphi(\kappa,\mu)$ is the solution of \eqref{fwmod} and $(\varphi'(\kappa,\mu))(\vartheta, \theta)$ 
is the solution of the variational problem \eqref{varphider}.
To approximate $J^h_{\kappa, \mu}$ at the optical parameters $\kappa(\tilde{\kappa})$ and $\mu(\tilde{\mu})$ corresponding to given $\tilde{\kappa}$ and $\tilde{\mu}$, 
we use the FEM with the piecewise linear basis $\{ \phi_n \}_{n=1}^N$ appearing in \eqref{discr_params} to evaluate $H'(\kappa,\mu)$, 
letting the perturbations $\vartheta$ and $\theta$ run in turns through $\{ \phi_n \}_{n=1}^N$. To be more precise, denoting the node values of the FEM approximations for $\varphi(\kappa, \mu)$, $(\varphi'(\kappa, \mu))(\vartheta,\theta)$ and $\theta$ by  $\{ \varphi_n \}$, $\{ \varphi'_n(\vartheta,\theta) \}$ and $\{ \theta_n \}$, respectively, we approximate
$$
\big(H'(\kappa,\mu) \big) (\vartheta, \theta) \approx 
\sum_{n=1}^N \big(h'_n(\kappa, \mu)\big)(\vartheta,\theta) \, \phi_n,
$$
where
\begin{align*}
\left( h_n'(\kappa,\mu) \right) (\vartheta, \theta) = \mu_n \varphi_n'(\vartheta,\theta) + \theta_n \varphi_n, \qquad n=1,\dots,N.
\end{align*}
Accordingly, the (approximate) Jacobian $J^h_{\kappa, \mu} = [ J^{h}_{\kappa} , J^{h}_{\mu} ]$ is formed as
\begin{align*}
J^{h}_{\kappa} &= \big[\big( h_i'(\kappa,\mu) \big) (\phi_j, 0)\big]_{i,j = 1}^N = \diag (\mu) \, J^{\varphi}_{\kappa}, \\
J^{h}_{\mu} &= \big[\big( h_i'(\kappa,\mu) \big) (0,\phi_j)\big]_{i,j = 1}^N = \diag (\mu) \, J^{\varphi}_{\mu} + \diag (\varphi).
\end{align*}
We still need to consider how to handle multiplications by $J^{\varphi}_{\kappa}$ and $J^{\varphi}_{\mu}$ matrix-freely.

Based on the weak formulation of \eqref{fwmod},
\begin{align*}
\int_{\Omega}\left( \kappa \nabla \varphi \cdot \nabla v + \mu \varphi v \right) dx  +
{\displaystyle \frac{1}{2}} \int_{\partial \Omega} \varphi v \, d S \nonumber = 2 \int_{\partial \Omega} \Phi v \, d S \qquad \text{for  all} \ v \in H^1(\Omega),
\end{align*}
we first form the FEM system matrix $K = K(\kappa,\mu) \in \R^{N \times N}$ and load vector $f = f(\Phi) \in \R^N$ corresponding to the basis $\{\phi_n\}_{n=1}^N$, and solve the node values $\varphi \in \R^N$ from the equation
\begin{equation*}
K \varphi = f. 
\end{equation*}
Next, 
we approximate the derivative $\left( \varphi'(\kappa,\mu) \right) (\vartheta, \theta)$ based on the variational problem \eqref{varphider}.
Since the left-hand side of \eqref{varphider} is identical to that in the weak formulation of \eqref{fwmod}, the FEM system matrix $K \in \R^{N \times N}$ stays the same. Letting the perturbations on the right-hand side of \eqref{varphider} run in turns through $\{ \phi_n \}_{n=1}^N$,
we get the load matrix $G(\varphi) \in \R^{N \times 2N} $ and the equation
\begin{align*}
K J^{\varphi}_{\kappa,\mu} = G \qquad {\rm or}
\qquad K \big[ J^{\varphi}_{\kappa}, \, J^{\varphi}_{\mu} \big] =
\big[ G^{(1)}, \, G^{(2)} \big],
\end{align*}
where
\begin{align*}
G^{(1)}_{i,j} = - \int_{\Omega} \phi_j \, \nabla \varphi \cdot \nabla \phi_i \, dx, \qquad
G^{(2)}_{i,j} = - \int_{\Omega} \phi_j \varphi \phi_i \, d x,
\end{align*}
for $i,j = 1,\ldots, N$.

Finally, instead of solving the huge system to form the Jacobian explicitly, we use the matrices $K$ and $G$ to implicitly operate 
on a given vector: 
multiplying $s = [(s^{(1)})^{\rm T} , (s^{(2)})^{\rm T}]^{\rm T} \in \R^{2N}$ by $J^h_{\tilde{\kappa}, \tilde{\mu}} \in \R^{N \times 2N}$ gives
\begin{align*}
J^h_{\tilde{\kappa}, \tilde{\mu}} s
&= J^{h}_{\tilde{\kappa}} s^{(1)} + J^{h}_{\tilde{\mu}} s^{(2)} \\
&= \diag (\mu) K^{-1}G^{(1)} \diag(\kappa ) s^{(1)} \\
& \quad + \left( \diag (\mu ) K^{-1}G^{(2)} + \diag (\varphi ) \right) \diag(\mu) s^{(2)} \\
&= \diag (\mu) K^{-1} \left( G^{(1)} \diag(\kappa) s^{(1)} + G^{(2)} \diag(\mu) s^{(2)} \right) \\
& \quad + \diag (\varphi) \diag (\mu) s^{(2)}.
\end{align*}
We also need to be able to multiply a given vector $t \in \R^{N}$ with the transpose $(J^h_{\tilde{\kappa}, \tilde{\mu}})^{\rm T} \in \R^{2N \times N}$. Since matrices $K$ and $G^{(2)}$ are symmetric, we get
\begin{align*}
J^{h}_{\tilde{\kappa}, \tilde{\mu}}{}^{\rm T} t
&= \left[ \begin{array}{c}
\!\! J^{h \rm T}_{\tilde{\kappa}}t \!\! \\[1mm]
\!\! J^{h \rm T}_{\tilde{\mu}}t \!\!
\end{array} \right]
= \left[ \begin{array}{c}
\!\! \diag(\kappa) (G^{(1)})^{\rm T} K^{-1} \diag (\mu) t \!\! \\[1mm]
\!\! \diag(\mu) \left( G^{(2)} K^{-1} \diag (\mu) + \diag (\varphi) \right) t \!\!
\end{array} \right] \\
&= \left[ \begin{array}{c c}
\!\! \, \diag (\kappa) & \mathrm{0} \!\! \\[1mm]
\!\! \mathrm{0} & \diag (\mu) \!\!
\end{array} \right]
\left( \left[ \begin{array}{c}
\!\! (G^{(1)})^{\rm T} \!\! \\[1mm]
\!\! G^{(2)} \!\!
\end{array} \right] K^{-1} \diag (\mu) t
+ \left[ \begin{array}{c}
\!\! \mathrm{0} \!\! \\[1mm]
\!\! \diag (\varphi) t \!\!
\end{array} \right] \right).
\end{align*}

Numerically speaking, multiplying a vector by the Jacobian $J^h_{\tilde{\kappa}, \tilde{\mu}}$ or its transpose is relatively cheap: one only needs to perform elementwise multiplications of vectors and matrix-vector multiplications with sparse matrices, and to operate with the inverse of the FEM system matrix $K$ on a vector. However, when the number of degrees of freedom and/or illuminations is very high, the sizes of the matrices $G^{(1)}$ and $G^{(2)}$ may become impractically large. In this case, one can continue the analysis to avoid forming the matrices $G^{(1)}$ and $G^{(2)}$ altogether, but the details are omitted here for brevity.

% Bibliography using bibtex %%%%%%%%%%%%%%%%%%%
\bibliographystyle{acm}
\bibliography{M105173}

\end{document}